\documentclass{amsart}
\usepackage{amsfonts}
\usepackage{amsmath,amssymb}
\usepackage{amsthm}
\usepackage{amscd}
\usepackage{graphics}
\usepackage{graphicx}

\theoremstyle{remark}{
\newtheorem{Def}{{\rm Definition}}
\newtheorem{Ex}{{\rm Example}}
\newtheorem{Rem}{{\rm Remark}}

}
\theoremstyle{plain}
{
\newtheorem{Cor}{Corollary}

\newtheorem{Thm}{Theorem}

}

\begin{document}
\title[Smooth maps onto regions respecting height functions]{Smooth maps onto regions in Euclidean spaces respecting height functions}
\author{Naoki kitazawa}

\address{Osaka Central Advanced Mathematical Institute (OCAMI) \\
3-3-138 Sugimoto, Sumiyoshi-ku Osaka 558-8585
TEL: +81-6-6605-3103
}
\email{naokikitazawa.formath@gmail.com}
\urladdr{https://naokikitazawa.github.io/NaokiKitazawa.html}
\maketitle
\begin{abstract}
One of interest of the author lies in the canonical projections of the unit spheres and generalizations. He is interested in the canonical projections of embedded manifolds in the higher dimensional Euclidean spaces, including height functions, especially, their construction and representations by explicit equations. This is a kind of fundamental studies on construction or global theory on singularity theory of differentiable maps, closely related to differential geometry, differential topology, and algebraic geometry. Rather than knowing existence, construction has been difficult in certain senses.

We construct smooth maps onto regions in Euclidean spaces locally modeled by height functions and show new examples. This is related to a pioneering study of the author, certain generalizations of the canonical projections of the unit spheres (, so-called {\it special generic} maps), and construction of them as projections of connected components of the zero sets of some smooth maps to some Euclidean spaces.


\end{abstract}
\section{Introduction.}
\label{sec:1}

The natural heights of the unit spheres and the canonical projections of the unit spheres are of fundamental smooth and real algebraic maps.

They are generalized to {\it height functions}, and the canonical projections of manifolds embedded into higher dimensional Euclidean spaces as smooth submanifolds (real algebraic manifolds). This is a kind of natural generalizations, respecting differential geometry and real algebraic geometry. As another kind of generalizations, {\it special generic} maps are of such ones, where singularities are of simplest ones for local extrema. Related studies have started in the story of higher dimensional versions of Morse function theory and related singularity theory of differentiable maps. For related classical and pioneering singularity theory, see \cite{thom, whitney}. For special generic maps, see \cite{burletderham, furuyaporto} and see \cite{saeki2} for related differential topological theory, with a related pioneering study \cite{saeki1}.

In global geometric theory of differentiable maps, existence of smooth maps of certain nice classes such as {\it fold} maps, which are canonically defined higher dimensional versions of fold maps, and special generic maps, are important. For fold maps, it is the existence theory via technique on differential equations (\cite{eliashberg1, eliashberg2} are pioneering studies). Explicit construction of fold maps global shapes and equations of which can be explicitly understood is another difficult problem. Special generic maps form a very specific subclass of the class of fold maps and characterizations of manifolds admitting such maps are nice and important problems in algebraic topology and differential topology. Construction of such maps is also not difficult: we can reconstruct naturally them from given smoothly immersed compact manifolds of codimension $0$ in the $n$-dimensional Euclidean space ${\mathbb{R}}^n$ ($\mathbb{R}:={\mathbb{R}}^1$), where we may reconstruct ones of more than one type.  

We use ${\pi}_{k,k_1}:{\mathbb{R}}^k \rightarrow {\mathbb{R}}^{k_1}$ for the projection to the first $k_1$ components ($1 \leq k_1<k$). Special generic maps of a certain simplest class have been reconstructed by the author again in the preprint \cite{kitazawa9}, respecting the canonical projection ${\pi}_{m+1,n} {\mid}_{S^m}$ of the $m$-dimensional unit sphere $S^m:=\{x=(x_1,\ldots x_{m+1}) \in {\mathbb{R}}^{m+1} \mid {\Sigma}_{j=1}^{m+1} {x_j}^2=1\}$, whose image is the disk $D^n:=\{x=(x_1,\ldots x_{n}) \in {\mathbb{R}}^{n} \mid {\Sigma}_{j=1}^{n} {x_j}^2 \leq 1\}$ and the natural projection of a naturally embedded manifold represented as the form ${\sharp}_{j=1}^l (S^{k_j} \times S^{m-k_j})$ ($1 \leq k_j \leq m-1$) of the connected sum defined in the smooth category, and their topological properties such as the cohomology rings of the manifolds of their domains are investigated. Motivated by the viewpoint of construction on explicit real algebraic sets and submanifolds represented as the zero sets of smooth maps into some Euclidean spaces, the author presented a pioneering study \cite{kitazawa5}. The preprint \cite{kitazawa6}, which is a kind of remark on \cite{kitazawa5}.

In the present paper, we construct smooth maps which may not be special generic motivated by the importance of construction. \cite{kitazawa8} is also a recent study on reconstructing special generic maps onto given regions, foliated by hypersurfaces and understood to be a further study to \cite{kitazawa5}.

In the next section, we present additional preliminaries and our main result, Theorems \ref{thm:1} and \ref{thm:2}, and related corollaries and examples.

\section{Our main result.}
\subsection{Preliminaries.}
We use $\mathbb{Z}$ for the set of all integers.

For a subspace $Y$ of a topological space $X$, the closure of $Y$ in $X$ is denoted by ${\overline{Y}}^X$.

A {\it proper} map $c:X \rightarrow Y$ means a map between topological spaces such that the preimage $c^{-1}(K)$ is compact for any compact subset $K \subset Y$.

For a tutorial of singularity theory of differentiable maps, see the textbook \cite{golubitskyguillemin}. A {\it singular} point of a smooth map $c:X \rightarrow Y$ and its {\it singular set} $S(c):=\{x \in X \mid \text{$x$ is a singular point of $c$.}\}$ are of fundamental and important notions. 
$c(S(c))$ is its {\it singular value set}. A {\it $C^r$-diffeomorphism} is a homeomorphism n the case $r=0$ and a homeomorphism with no singular point in the case of $r \geq 1$ being an positive integer or $r=\infty$: in the case $r=\infty$ it is a {\it diffeomorphism}.
We use "{\it critical}" instead of "singular" in the case the space of the target is a $1$-dimensional manifold. Morse functions appear as important explicit functions. See \cite{milnor1, milnor2}. {\it Fold} maps are defined to be smooth maps between manifolds with no boundary which are around singular points represented as the products of Morse functions and the identity map some open set in a Euclidean space. For this, see \cite{golubitskyguillemin} again and presented papers on fold maps.

We define important notions on equivalence of smooth maps.
\begin{Def}
For smooth maps $c_1:X_1 \rightarrow Y_1$ and $c_2:X_2 \rightarrow Y_2$, we have the following equivalence relation.
\begin{enumerate}
\item $c_1$ and $c_2$ are {\it $(C^r,\mathcal{A})$-equivalent} or equivalently,  $c_1$ {\rm (}$c_2${\rm )} is {\it $(C-r,\mathcal{A})$-equivalent} to $c_2$ {\rm (}resp. $c_1${\rm )}, if there exists a pair $({\phi}_X:X_1 \rightarrow X_2,{\phi}_Y:Y_1 \rightarrow Y_2)$ of $C^r$-diffeomorphisms satisfying ${\phi}_Y \circ c_1=c_2 \circ {\phi}_X$. In the case $r=\infty$, we replace $(C^r,\mathcal{A})$ by $\mathcal{A}$.
\item If in the definition above, $Y_1=Y_2$ and ${\phi}_Y$ can be chosen to be the identity map, then $\mathcal{A}$ is replaced by {\it $\mathcal{R}$}.
\end{enumerate}
\end{Def}

A {\it height function} of an $m$-dimensional submanifold $X$ smoothly embedded in ${\mathbb{R}}^k$ ($m \leq k$) is ${\pi}_{k,1} {\mid}_{X}$. 

A {\it special generic} map is a fold map from an $m$-dimensional manifold to $n$-dimensional manifold ($m \geq n$) which are around each singular point represented as the product of the height function $h_{D^{m-n+1}}:D^{m-n+1} \rightarrow \mathbb{R}$ ($h_{D^{m-n+1}}(x):=h_{D^{m-n+1}}((x_1,\ldots x_{m-n+1}))={\Sigma}_{j=1}^{m-n+1} {x_j}^2$) and the identity map on some open set in the Euclidean space ${\mathbb{R}}^{n-1}$ for suitable local coordinates.

A {\it real algebraic} manifold $X$ means a submanifold with no boundary embedded in ${\mathbb{R}}^{m+k}$ ($k>0$) is a union of connected components of the zero set of a real polynomial map $F_X$ from ${\mathbb{R}}^{m+k}$ to ${\mathbb{R}}^{k}$ whose restriction $F_X {\mid_X}$ has no singular point. A {\it real algebraic} map is such a map ${\pi}_{m+k,n} {\mid}_X$.

For related real algebraic geometry, refer to \cite{kollar, lellis} and see also the textbook \cite{bochnakcosteroy}.

An {\it affine transformation} of ${\mathbb{R}}^k$ means a homeomorphism (diffeomorphism) of ${\mathbb{R}}^k$ represented as the composition of finitely many linear transformations and parallel transformations on ${\mathbb{R}}^k$.


Hereafter, $\{S_{\lambda}\}_{\lambda \in \Lambda}$ denotes a family of mutually disjoint smooth ($n-1$)-dimensional manifolds (hypersurfaces) with no boundary of ${\mathbb{R}}^n$.

\subsection{One of our main result and related examples.}
This is a generalization of \cite{kitazawa5} and a preprint \cite{kitazawa6} and related to a recent study \cite{kitazawa7}, where we do not need to understand these non-trivial mathematical arguments with facts.

\begin{Thm}
\label{thm:1}
Let $h_X:X \rightarrow \mathbb{R}$ denote a proper height function ${\pi}_{m_2,1} {\mid}_X$ on an $m_1$-dimensional smooth connected compact submanifold of ${\mathbb{R}}^{m_2}$ {\rm (}$1 \leq m_1<m_2${\rm )} with no boundary whose critical value set $h_X(S(h_X))$ is closed and discrete in $\mathbb{R}$ and represented as $\{s_{{\lambda}_{j}} \in \mathbb{R} \mid {\lambda}_{j} \in {\Lambda}_{h_X} \}$ for some subset of integers ${\Lambda}_{h_x}$ with the following.
\begin{itemize}
\item ${\Lambda}_{h_x}$ has the minimum integer $m_{{\Lambda}_{h_x}}$.
\item For any distinct two elements ${\lambda}_{j_1} \in {\Lambda}_{h_x}$ and ${\lambda}_{j_2} \in {\Lambda}_{h_x}$ satisfying ${\lambda}_{j_1}<{\lambda}_{j_2}$, every integer ${\lambda}_{j_1}<{\lambda}_{j_1,j_2}<{\lambda}_{j_2}$ is an element of ${\Lambda}_{h_X}$.
\item $s_{{\lambda}_{j_1}}<s_{{\lambda}_{j_2}}$ for any pair ${\lambda}_{j_1}<{\lambda}_{j_2}$.
\end{itemize}
Suppose also that a family $\{S_{\lambda}\}_{\lambda \in {\Lambda}_S}$ of mutually disjoint smooth {\rm (}$n-1${\rm )}-dimensional manifolds {\rm (}hypersurfaces{\rm )} with no boundary of ${\mathbb{R}}^n$ located in some bounded set of ${\mathbb{R}}^n$, an integer-valued function $m_{S}:{\Lambda}_S \rightarrow {\Lambda}_{h_x}$ on ${\Lambda}_S$, and a map $m_{S,0}$, are given, as follows.
\begin{enumerate}
\item \label{thm:1.1} There exists a unique unbounded connected region $R_{{\lambda}_{\rm u}}$ of ${\mathbb{R}}^n-{\sqcup}_{\lambda \in \Lambda} S_{\lambda}$. Furthermore, the boundary of the closure ${\overline{R_{{\lambda}_{\rm u}}}}^{{\mathbb{R}}^n}$ is represented by one hypersurface $S_{{\lambda}_{o}}$ {\rm (}${\lambda}_{o} \in {\Lambda}_S${\rm )} which is also connected and satisfies $m_{S}({\lambda}_{o})=m_{{\Lambda}_{h_x}}$.
\item  \label{thm:1.2}
The boundary of the closure ${\overline{R_{{\Lambda}_i}}}^{{\mathbb{R}}^n}$ of each open and bounded connected region $R_{{\Lambda}_i}$ of ${\mathbb{R}}^n-{\sqcup}_{\lambda \in {\Lambda}_S} S_{\lambda}$ is represented by the form ${\sqcup}_{{\lambda}_{i} \in {\Lambda}_i} S_{{\lambda}_i}$ and for this, the image $m_{S}({\Lambda}_i)$ consists of at most $2$ numbers and in the case of a two-element set, the difference of the two numbers is $\pm 1$. We can find a one-to one correspondence between each open and bounded connected region $R_{{\Lambda}_i}$ of ${\mathbb{R}}^n-{\sqcup}_{\lambda \in {\Lambda}_S} S_{\lambda}$ and such a set ${\Lambda}_i \subset \Lambda$.
\item  \label{thm:1.3}
 If the boundaries of the closures ${\overline{R_{{\Lambda}_{i_1}}}}^{{\mathbb{R}}^n}$ and ${\overline{R_{{\Lambda}_{i_2}}}}^{{\mathbb{R}}^n}$ of distinct open and bounded connected regions $R_{{\lambda}_{i_1}}$ and $R_{{\lambda}_{i_2}}$ of ${\mathbb{R}}^n-{\sqcup}_{\lambda \in \Lambda} S_{\lambda}$ with $m_{S}({\Lambda}_{i_1})$ and $m_{S}({\Lambda}_{i_2})$ being two-point sets intersect, then the images $m_{S}({\Lambda}_{i_1})$ and $m_{S}({\Lambda}_{i_2})$ are distinct.
\item \label{thm:1.4}
 There exists a map $m_{S,0}$ from the set of all open and bounded connected regions $R_{{\Lambda}_i}$ of ${\mathbb{R}}^n-{\sqcup}_{\lambda \in {\Lambda}_S} S_{\lambda}$ with $m_{S}({\Lambda}_{i})$ being one-point sets to $\{1,2\} \subset \mathbb{Z}$. Note that such distinct two open and bounded connected regions $R_{{\Lambda}_{i_1}}$ and $R_{{\Lambda}_{i_2}}$ of ${\mathbb{R}}^n-{\sqcup}_{\lambda \in {\Lambda}_S} S_{\lambda}$ whose closures ${\overline{R_{{\Lambda}_{i_1}}}}^{{\mathbb{R}}^n}$ and ${\overline{R_{{\Lambda}_{i_2}}}}^{{\mathbb{R}}^n}$ cannot intersect.
\item \label{thm:1.5} If the boundaries of the closures ${\overline{R_{{\Lambda}_{i_1}}}}^{{\mathbb{R}}^n}$ and ${\overline{R_{{\Lambda}_{i_2}}}}^{{\mathbb{R}}^n}$ of distinct two open and bounded connected regions $R_{{\Lambda}_{i_1}}$ and $R_{{\Lambda}_{i_2}}$ of ${\mathbb{R}}^n-{\sqcup}_{\lambda \in \Lambda} S_{\lambda}$ with $m_{S}({\Lambda}_{i_1})$ being a two-point set and $m_{S}({\Lambda}_{i_2})$ being a one-point set intersect and the unique element of $m_{S}({\Lambda}_{i_1})-m_{S}({\Lambda}_{i_2})$ is greater {\rm (}smaller{\rm )} than the unique element of $m_{S}({\Lambda}_{i_2})$, then $m_{S,0}(R_{{\Lambda}_{i_2}})$ is $1$ {\rm (}resp. $2${\rm )}.
\item  \label{thm:1.6}

The union of the closures ${\overline{R_{{\Lambda}_i}}}^{{\mathbb{R}}^n}$ of all open and bounded connected regions $R_{{\Lambda}_i}$ of ${\mathbb{R}}^n-{\sqcup}_{\lambda \in {\Lambda}_S} S_{\lambda}$ and the function $h_X$ satisfy at least one of the following. Hereafter, if $h_X$ has the maximum, then ${\Lambda}_{h_x}$ has the maximum integer $m_{{\Lambda}_{h_x},{\rm M}}$ of course.
\begin{enumerate}
\item \label{thm:1.6.1}
$h_X$ has the minimum and the maximum. The union of the closures ${\overline{R_{{\Lambda}_i}}}^{{\mathbb{R}}^n}$ of all open and bounded connected regions $R_{{\Lambda}_i}$ of ${\mathbb{R}}^n-{\sqcup}_{\lambda \in {\Lambda}_S} S_{\lambda}$ satisfying at least one of the following is connected.
\begin{enumerate}
\item $m_{S}({\Lambda}_i) \neq \{m_{{\Lambda}_{h_x}}\}, \{m_{{\Lambda}_{h_x},{\rm M}}\}$.
\item $m_{S}({\Lambda}_i)=\{m_{{\Lambda}_{h_x}}\}$ and $m_{S,0}(\overline{R_{{\Lambda}_i}})=2$.
\item $m_{S}({\Lambda}_i)=\{m_{{\Lambda}_{h_x},{\rm M}}\}$ and $m_{S,0}(\overline{R_{{\Lambda}_i}})=1$.
\end{enumerate}
\item \label{thm:1.6.2}
$h_X$ has the maximum and does not have the minimum. The union of the closures ${\overline{R_{{\Lambda}_i}}}^{{\mathbb{R}}^n}$ of all open and bounded connected regions $R_{{\Lambda}_i}$ of ${\mathbb{R}}^n-{\sqcup}_{\lambda \in {\Lambda}_S} S_{\lambda}$ satisfying at least one of the following is connected.
\begin{enumerate}
\item $m_{S}({\Lambda}_i) \neq \{m_{{\Lambda}_{h_x}}\}$.
\item $m_{S}({\Lambda}_i)=\{m_{{\Lambda}_{h_x}}\}$ and $m_{S,0}(\overline{R_{{\Lambda}_i}})=2$.
\end{enumerate}
\item \label{thm:1.6.3}
$h_X$ has the minimum and does not have the maximum. The union of the closures ${\overline{R_{{\Lambda}_i}}}^{{\mathbb{R}}^n}$ of all open and bounded connected regions $R_{{\Lambda}_i}$ of ${\mathbb{R}}^n-{\sqcup}_{\lambda \in {\Lambda}_S} S_{\lambda}$ satisfying at least one of the following is connected.
\begin{enumerate}
\item $m_{S}({\Lambda}_i) \neq \{m_{{\Lambda}_{h_x}}\}$.
\item $m_{S}({\Lambda}_i)=\{m_{{\Lambda}_{h_x},{\rm M}}\}$ and $m_{S,0}(\overline{R_{{\Lambda}_i}})=1$.
\end{enumerate}
\item \label{thm:1.6.4}
 $h_X$ does not have the minimum or the maximum.
\end{enumerate}
\end{enumerate}
In this setting, we have the following objects. 

\begin{enumerate}
\setcounter{enumi}{6}
\item \label{thm:1.7} A smooth function $c_{\{S_{\lambda}\}_{\lambda \in {\Lambda}_S}}$ on ${\mathbb{R}}^n$ such that ${c_{\{S_{\lambda}\}_{\lambda \in {\Lambda}_S}}}^{-1}(s_{{\lambda}_{j}})={\sqcup}_{{\lambda}_ {j^{\prime},j} \in \{{{\lambda}^{\prime}}_{j^{\prime}} \mid m_{S}({{\lambda}^{\prime}}_{j^{\prime}})={\lambda}_j\}} S_{{\lambda}_ {j^{\prime},j}}$ and that $c_{\{S_{\lambda}\}_{\lambda \in {\Lambda}_S}} {\mid}_{S_{\lambda}}$ has no critical point for any $\lambda \in {\Lambda}_S$ .
\item \label{thm:1.8} An $m$-dimensional smooth connected submanifold $X_{{\Lambda}_{h_x},\{S_{\lambda}\}_{\lambda \in {\Lambda}_S},m_{S}, c_{\{S_{\lambda}\}_{\lambda \in {\Lambda}_S}}}$ of ${\mathbb{R}}^{m_2+n-1}$  with no boundary, where $m:=m_1+n-1$. More precisely, $X_{{\Lambda}_{h_x},\{S_{\lambda}\}_{\lambda \in {\Lambda}_S},m_{S}, c_{\{S_{\lambda}\}_{\lambda \in {\Lambda}_S}}}$ can be defined as the set canonically by the relation $X_{{\Lambda}_{h_x},\{S_{\lambda}\}_{\lambda \in {\Lambda}_S},m_{S}, c_{\{S_{\lambda}\}_{\lambda \in {\Lambda}_S}}} \bigcap {{\pi}_{m_2+n-1,n}}^{-1}(p)=\{p\} \times {h_X}^{-1}(c_{\{S_{\lambda}\}_{\lambda \in {\Lambda}_S}}(p))$ {\rm (}$p \in {\mathbb{R}}^n${\rm )}.
\end{enumerate}
\end{Thm}
\begin{proof}
STEP 1 On $c_{\{S_{\lambda}\}_{\lambda \in {\Lambda}_S}}:{\mathbb{R}}^n \rightarrow \mathbb{R}$ of (\ref{thm:1.7}). \\

We can construct a smooth $c_{\{S_{\lambda}\}_{\lambda \in {\Lambda}_S}}:{\mathbb{R}}^n \rightarrow \mathbb{R}$ of (\ref{thm:1.1}) with the following, by existence theory of smooth functions based on so-called partitions of the unity. The conditions (\ref{thm:1.1}, \ref{thm:1.2}, \ref{thm:1.3}, \ref{thm:1.4}, \ref{thm:1.5}) are also respected. 
\begin{itemize}
\item ${c_{\{S_{\lambda}\}_{\lambda \in {\Lambda}_S}}}^{-1}(s_{{\lambda}_{j}})={\sqcup}_{{\lambda}_ {j^{\prime},j} \in \{{{\lambda}^{\prime}}_{j^{\prime}} \mid m_{S}({{\lambda}^{\prime}}_{j^{\prime}})={\lambda}_j\}} S_{{\lambda}_ {j^{\prime},j}}$.
\item If for an open and bounded connected region $R_{{\Lambda}_i}$ of ${\mathbb{R}}^n-{\sqcup}_{\lambda \in {\Lambda}_S} S_{\lambda}$, the boundary of its closure ${\overline{R_{{\Lambda}_i}}}^{{\mathbb{R}}^n}$ is represented by the form ${\sqcup}_{{\lambda}_{i} \in {\Lambda}_i} S_{{\lambda}_i}$ and  the image $m_{S}({\Lambda}_i)$ consists of exactly $2$ numbers ${\lambda}_{i,1}<{\lambda}_{i,2}$, then the image $c_{\{S_{\lambda}\}_{\lambda \in {\Lambda}_S}}(R_{{\Lambda}_i})$ is $\{s \mid s_{{\lambda}_{i,1}}<s<s_{{\lambda}_{i,2}}\}$.
\item If for an open and bounded connected region $R_{{\Lambda}_i}$ of ${\mathbb{R}}^n-{\sqcup}_{\lambda \in {\Lambda}_S} S_{\lambda}$, the boundary of its closure ${\overline{R_{{\Lambda}_i}}}^{{\mathbb{R}}^n}$ is represented by the form ${\sqcup}_{{\lambda}_{i} \in {\Lambda}_i} S_{{\lambda}_i}$ and  the image $m_{S}({\Lambda}_i)$ consists of exactly $1$ number ${\lambda}_{i}$ with $m_{S,0}(R_{{\Lambda}_i})=1$, then the image $c_{\{S_{\lambda}\}_{\lambda \in {\Lambda}_S}}(R_{{\Lambda}_i})$ is
a connected subset of $\{s \mid s<s_{{\lambda}_{i}}\}$ and  if $s_{{\lambda}_{i}-1}$ exists in addition, then the minimum of $c_{\{S_{\lambda}\}_{\lambda \in {\Lambda}_S}}(R_{{\Lambda}_i})$ is strictly greater than $s_{{\lambda}_{i}-1}$.
\item If for an open and bounded connected region $R_{{\Lambda}_i}$ of ${\mathbb{R}}^n-{\sqcup}_{\lambda \in {\Lambda}_S} S_{\lambda}$, the boundary of its closure ${\overline{R_{{\Lambda}_i}}}^{{\mathbb{R}}^n}$ is represented by the form ${\sqcup}_{{\lambda}_{i} \in {\Lambda}_i} S_{{\lambda}_i}$ and the image $m_{S}({\Lambda}_i)$ consists of exactly $1$ number ${\lambda}_{i}$ with $m_{S,0}(R_{{\Lambda}_i})=2$, then the image $c_{\{S_{\lambda}\}_{\lambda \in {\Lambda}_S}}(R_{{\Lambda}_i})$ is
a connected subset of $\{s \mid s>s_{{\lambda}_{i}}\}$ and  if $s_{{\lambda}_{i}+1}$ exists in addition, the maximum of $c_{\{S_{\lambda}\}_{\lambda \in {\Lambda}_S}}(R_{{\Lambda}_i})$ is strictly smaller than $s_{{\lambda}_{i}+1}$.
\end{itemize}
(Small) closed tubular neighborhood $N(S_{\lambda}) \subset {\mathbb{R}}^n$ of $S_{\lambda}$ in ${\mathbb{R}}^n$ is represented as the product bundle $S_{\lambda} \times \{s \mid s_{\lambda}-{\epsilon}_{\lambda} \leq s \leq  s_{\lambda}+{\epsilon}_{\lambda}\}$
 over $S_{\lambda}$, identified via the map mapping $x$ to $(x,0)$, with $S_{\lambda} \times \{0\}$, where ${\epsilon}_{\lambda}$ is a suitable small real number. We can also have $c_{\{S_{\lambda}\}_{\lambda \in {\Lambda}_S}}$ in such a way that the following hold.
\begin{itemize}
\item $c_{\{S_{\lambda}\}_{\lambda \in {\Lambda}_S}} {\mid}_{N(S_{\lambda})}$ has no critical point.
\item $c_{\{S_{\lambda}\}_{\lambda \in {\Lambda}_S}} {\mid}_{N(S_{\lambda})}$ is regarded to be the projection to the component $\{t \mid s_{\lambda}-{\epsilon}_{\lambda} \leq s_{\lambda} \leq  s_{\lambda}+{\epsilon}_{\lambda}\}$ naturally.
\item On the unbounded connected component of ${\mathbb{R}}^n-N(S_{{\lambda}_{o}})$, the values of $c_{\{S_{\lambda}\}_{\lambda \in {\Lambda}_S}}$ are always smaller than $s_{\lambda}-{\epsilon}_{\lambda}$.
\end{itemize}

STEP 2  $X_{{\Lambda}_{h_x},\{S_{\lambda}\}_{\lambda \in {\Lambda}_S},m_{S}, c_{\{S_{\lambda}\}_{\lambda \in {\Lambda}_S}}}$ in (\ref{thm:1.8}), defined by the relation\\ $X_{{\Lambda}_{h_x},\{S_{\lambda}\}_{\lambda \in {\Lambda}_S},m_{S}, c_{\{S_{\lambda}\}_{\lambda \in {\Lambda}_S}}} \bigcap {{\pi}_{m_2+n-1}}^{-1}(p)=\{p\} \times {h_X}^{-1}(c_{\{S_{\lambda}\}_{\lambda \in {\Lambda}_S}}(p))$\\ is an $m$-dimensional smooth connected submanifold with no boundary. \\

On the preimage ${{{\pi}_{m_2+n-1,n}} {\mid}_{X_{{\Lambda}_{h_x},\{S_{\lambda}\}_{\lambda \in {\Lambda}_S},m_{S}, c_{\{S_{\lambda}\}_{\lambda \in {\Lambda}_S}}}}}^{-1}(N(S_{\lambda}))$, the restriction ${\pi}_{} {\mid}_{X_{{\Lambda}_{h_x},\{S_{\lambda}\}_{\lambda \in {\Lambda}_S},m_{S}, c_{\{S_{\lambda}\}_{\lambda \in {\Lambda}_S}}}}$ is regarded to be a smooth map $\mathcal{A}$-equivalent to the product map of the restriction $h_X {\mid}_{{h_X}^{-1}(\{s \mid s_{\lambda}-{\epsilon}_{\lambda} \leq s \leq  s_{\lambda}+{\epsilon}_{\lambda}\})}$ and the identity map on $S_{\lambda}$.

On the non-empty preimage ${{\pi}_{m_2+n-1,n} {\mid}_{X_{{\Lambda}_{h_x},\{S_{\lambda}\}_{\lambda \in {\Lambda}_S},m_{S}, c_{\{S_{\lambda}\}_{\lambda \in {\Lambda}_S}}}}}^{-1}(R_{{\Lambda}_i})$, by the assumption that $h_X$ is proper and Ehresmann fibration theorem, it is regarded to be the projection of a product bundle over $R_{{\Lambda}_i}$.

Connectedness of $X_{{\Lambda}_{h_x},\{S_{\lambda}\}_{\lambda \in {\Lambda}_S},m_{S}, c_{\{S_{\lambda}\}_{\lambda \in {\Lambda}_S}}}$ is due to the assumption that the manifold $X$ of the domain $h_X$ is connected and the condition (\ref{thm:1.6}) with the property of $c_{\{S_{\lambda}\}_{\lambda \in {\Lambda}_S}}$. We add exposition on (\ref{thm:1.6}). In the conditions (\ref{thm:1.6.1}, \ref{thm:1.6.2}, \ref{thm:1.6.3}), each of the unions of the closures ${\overline{R_{{\Lambda}_i}}}^{{\mathbb{R}}^n}$ of all open and bounded connected regions $R_{{\Lambda}_i}$ of ${\mathbb{R}}^n-{\sqcup}_{\lambda \in {\Lambda}_S} S_{\lambda}$ satisfying the listed conditions and the image ${\pi}_{m_2+n-1,n}(X_{{\Lambda}_{h_x},\{S_{\lambda}\}_{\lambda \in {\Lambda}_S},m_{S}, c_{\{S_{\lambda}\}_{\lambda \in {\Lambda}_S}}})$ coincide. In the condition (\ref{thm:1.6.4}), the image and the union of the closures ${\overline{R_{{\Lambda}_i}}}^{{\mathbb{R}}^n}$ of all open and bounded connected regions $R_{{\Lambda}_i}$ of ${\mathbb{R}}^n-{\sqcup}_{\lambda \in {\Lambda}_S} S_{\lambda}$ coincide.

This completes the proof.
\end{proof}
We introduce several classes of fold maps. For a fold map $f:M \rightarrow N$ from an $m$-dimensional manifold $M$ to an $n$-dimensional manifold $N$, its singular set is an ($n-1$)-dimensional smooth (regular) submanifold of $M$ with no boundary and the restriction of $f$ to the singular . \\set is a smooth immersion. If the restriction is an embedding, then it is an {\it image simple} fold map. 
Special generic maps explained above are regarded to be image simple fold maps. If an image simple fold map is $\mathcal{A}$-equivalent to a fold map the image of whose singular set is of the form ${\sqcup}_{i=1}^l \{x=(x_1,\ldots x_n) \mid {\Sigma}_{j=1}^{n} {x_j}^2=i^2\} \subset {\mathbb{R}}^n$ with $n \geq 2$ ($l>0$), then it is a {\it round} fold map. If the image of the singular set of  a round fold map is of the form ${\sqcup}_{i=1}^l \{x=(x_1,\ldots x_n) \mid {\Sigma}_{j=1}^{n} {x_j}^2={r_i}^2\} \subset {\mathbb{R}}^n$ with $n \geq 2$ and an increasing sequence $\{r_i\}_{i=1}^l$ of positive real numbers ($l>0$), then it is a {\it normalized} round fold map. If $r_i=i$ before, then it is a {\it completely normalized} round fold map.

Every round fold map on an $m$-dimensional closed manifold $c:X \rightarrow {\mathbb{R}}^n$ ($m \geq n$) is $\mathcal{A}$-equivalent to a completely normalized one $c_0:X \rightarrow {\mathbb{R}}^n$ and we have its {\it page function}, defined as the restriction of the completely normalized round fold map to the preimage ${c_0}^{-1}(\{tx_0=t(x_{0,1} \ldots x_{0,n}) \mid t \geq \frac{1}{2}\})$ of a straight half line $\{tx_0=t(x_{0,1} \ldots x_{0,n}) \mid t \geq \frac{1}{2}\}$ ($x_0 \in S^{n-1}$) uniquely, up to $\mathcal{R}$-equivalence: this is independent of choice of $(c_0,x_0)$ (\cite{kitazawasaeki}) and depends only on the $\mathcal{A}$-equivalence class of a given round fold map.

A round fold map has been formulated by the author in \cite{kitazawa1, kitazawa2, kitazawa3} and pioneering construction is in \cite{kitazawa4}. Image simple fold maps have been formally defined in \cite{sadykovsaeki}, where such a class has been important in \cite{saeki3} for example.

By the definitions and rules, we immediately have Corollary \ref{cor:1}.
\begin{Cor}
\label{cor:1}
If the height function is a Morse function in Theorem \ref{thm:1}, then we have an image simple fold map ${\pi}_{m_2+n-1,n} {\mid}_{X_{{\Lambda}_{h_x},\{S_{\lambda}\}_{\lambda \in {\Lambda}_S},m_{S}, c_{\{S_{\lambda}\}_{\lambda \in {\Lambda}_S}}}}$.
\end{Cor}
Corollary \ref{cor:2} is shown by applying fundamental arguments on approximation and real algebraic geometry.
\begin{Cor}
\label{cor:2}
In Theorem \ref{thm:1}, suppose that $h_X$ is a real algebraic function ${\pi}_{m_2,1} {\mid}_{X}$.
In this case, for any non-negative integer $r \geq 0$, we have another real algebraic map ${\pi}_{m_2+n-1,n} {\mid}_{X_{{\rm R.A.},r,{\Lambda}_{h_x},\{S_{\lambda}\}_{\lambda \in {\Lambda}_S},m_{S}, c_{\{S_{\lambda}\}_{\lambda \in {\Lambda}_S}}}}$ which is real algebraic and which is $(C^r,\mathcal{A})$-equivalent to the original smooth map ${\pi}_{m_2+n-1,n} {\mid}_{X_{{\Lambda}_{h_x},\{S_{\lambda}\}_{\lambda \in {\Lambda}_S},m_{S}, c_{\{S_{\lambda}\}_{\lambda \in {\Lambda}_S}}}}$. 
In addition, in the case $X$ is the zero set of some real polynomial map from ${\mathbb{R}}^{m_2}$ to ${\mathbb{R}}^{m_2-m_1}$, we can also choose $X_{{\rm R.A.},r,{\Lambda}_{h_x},\{S_{\lambda}\}_{\lambda \in {\Lambda}_S},m_{S}, c_{\{S_{\lambda}\}_{\lambda \in {\Lambda}_S}}}$ to be the zero set of some real polynomial map from ${\mathbb{R}}^{m_2+n-1}$ to ${\mathbb{R}}^{m_2-m+n-1}$.
\end{Cor}
\begin{proof}
We can approximate the presented function $c_{\{S_{\lambda}\}_{\lambda \in {\Lambda}_S}}$ in the so-called $C^r$-topology and have a real polynomial function $c_{{\rm R.A.},r,\{S_{\lambda}\}_{\lambda \in {\Lambda}_S}}$ satisfying
 the conditions presented in Proof of Theorem \ref{thm:1}. 
For this, apply arguments in \cite{lellis} (, especially \cite[Proposition 1.3.7]{lellis} and arguments there,) and apply also ones of \cite[Sepcial Case 5]{kollar}.
We can have a desired objects in the similar way.

We can also discuss the case $X$ is the zero set of some real polynomial map from ${\mathbb{R}}^{m_2}$ to ${\mathbb{R}}^{m_2-m_1}$.

This completes the proof.
\end{proof}
We present an explicit case for Corollaries \ref{cor:1} and \ref{cor:2} as Example \ref{ex:1}, where we  do not need to understand the content rigorously.

For this, we explain Reeb digraphs of Morse(-Bott) functions on compact manifolds shortly.
Given such a function $c:X \rightarrow \mathbb{R}$ on a compact and connected manifold $X$. If the boundary of $X$ is non-empty, then we pose the condition that the level set $c^{-1}(m_c)$ of its minimum $m_c$ and its boundary coincide.
The Reeb digraph $\overrightarrow{R_c}$ can be defined as the digraph as follows.
\begin{itemize}
\item The underlying space is the quotient space where the equovalence relation on $X$ is as follows. Two points of $X$ are equivalent if and only if they are in a same connected component of a same level set $c^{-1}(t)$.
\item We define the quotient map $q_c:X \rightarrow \overrightarrow{R_c}$. A point $v$ is a vertex of the digraph if and only if ${q_c}^{-1}(v)$ contains some critical point of $c$.
\item We also have the unique continuous function $\bar{c}:\overrightarrow{R_c} \rightarrow \mathbb{R}$ satisfying $c=\bar{c} \circ q_c$. Each edge $e$ connects distinct two vertices $v_{e,1}$ and $v_{e,2}$ and it is oriented from $v_{e,1}$ to $v_{e,2}$ if and only if $\bar{c}(v_{e,1})<\bar{c}(v_{e,2})$.

\end{itemize}
For such graphs, see \cite{reeb} as a pioneering study. See \cite{izar} for the case of Morse functions. See \cite{martinezalfaromezasarmientooliveira} for the  case of Morse-Bott functions.
\begin{Ex}
\label{ex:1}

According to \cite{kitazawasaeki}, round fold maps on closed and connected manifolds of dimension $m$ into ${\mathbb{R}}^{m-1}$ with $m \geq 5$ are classified by $\mathcal{R}$-equivalence classes of their page Morse functions, up to $\mathcal{A}$-equivalence. 
In \cite{kitazawa7}, each of these round fold maps whose  page functions $c_{\rm P}:X_{\rm P} \rightarrow \mathbb{R}$ are ones on compact, connected and orientable surfaces and whose Reeb digraphs are naturally realized as the compositions
 $\bar{c_{\rm P}}$ of embeddings $e_{\rm P}:\overrightarrow{R_{{\rm P}}} \rightarrow {\mathbb{R}}^3$ with ${\pi}_{3,1}$ is regarded to be a normalized one and represented as a real algebraic map. This is also regarded to be an explicit case of Theorem \ref{thm:1}. 

Note that according to \cite{izar,martinezalfaromezasarmientooliveira}, $\mathcal{R}$-equivalence classes of Morse-Bott functions on compact, connected and orientable surfaces such that the values at distinct connected components of their critical sets are distinct are determined by natural isomorphism classes of their Reeb digraphs.
\end{Ex}
\subsection{Another main result and related examples.}
This is a generalization of \cite{kitazawa8}. 
\begin{Def}
\label{def:2}
Let $m_1<m_2$ be positive integers. Given the following objects.
\begin{enumerate}
\item  (A closed region $E_{F_I} \subset {\mathbb{R}}^n$ foliated by a smooth function $F_I:{\mathbb{R}}^n \times I \rightarrow \mathbb{R}$.)
Let $I$ be the interval of the form $I=\{t \mid t \geq a_0\}$ or $I:=\{t \mid a_1 \leq t \leq a_2\}$ with $a_1<a_2$.
Let $F_I:{\mathbb{R}}^n \times I \rightarrow \mathbb{R}$ be a smooth function satisfying the following.
\begin{itemize}
\item We define $F_{I,t}:{\mathbb{R}}^n \rightarrow \mathbb{R}$ by $F_{I,t}(x):=F_I(x,t)$ for $t \in I$.
\item ${F_{I,t_1}}^{-1}(0)$ and ${F_{I,t_2}}^{-1}(0)$ are disjoint for distinct two numbers $t_1$ and $t_2$.
\item $E_{F_I}={\bigcup}_{t \in I} {F_{I,t}}^{-1}(0)$ and is non-empty.
\item The restriction of $F_{I,a_j}$ to ${\mathbb{R}}^n$ has no critical point for $j=0,1,2$.
\end{itemize}
\item A real number $b$ and a smooth function $s_I:I \rightarrow \mathbb{R}$ satisfying $I=\{t \mid s_I(t) \geq b\}$. We also assume that $\{t \mid s_I(t)>b\}=I-\{a_0\}$ in the case $I=\{t \mid t \geq a_0\}$ and that $\{t \mid s_I(t)>b\}=I-\{a_1,a_2\}$ in the case $I:=\{t \mid a_1 \leq t \leq a_2\}$. The value of the derivative of the function at each $a_j$ is not $0$.
\item (A proper height function $h_X:={{\pi}_{m_2,1}} {\mid}_{X}:X \rightarrow {\mathbb{R}}$ on an $m_1$-dimensional submanifold of ${\mathbb{R}}^{m_2}$.)
A proper height function $h_X:={{\pi}_{m_2,1}} {\mid}_{X}:X \rightarrow {\mathbb{R}}$ on an $m_1$-dimensional submanifold of ${\mathbb{R}}^{m_2}$ whose values are always greater than or equal to $b$, whose image
 $h_X(X)$ contains $s_I(I)$ and the image of whose singular set is closed and discrete in $\mathbb{R}$.

\end{enumerate}  
We call the data $(F_I,E_{F_I},s_I)$ an {\it $n$-admissible foliated base} or an {\it $n$-AFB}. We call $h_X$ an {\it {\rm (}$m_1,m_2${\rm )}-admissible} function on $(F_I,E_{F_I},s_I)$.
\end{Def}

\begin{Thm}
\label{thm:2}
Suppose that an $n$-AFB $(F_I,E_{F_I},s_I)$ and an $(m_1,m_2)$-admissible function $h_X$ on $(F_I,E_{F_I},s_I)$ are given.

We can define a subset $X_{(F_I,E_{F_I},s_I,h_X),m}$ of ${\mathbb{R}}^{m+m_2-m_1+1}$ with $m:=m_1+n-1$ by the relation
$X_{(F_I,E_{F_I},s_I,h_X),m} \bigcap {{\pi}_{m_2+n,n+1}}^{-1}(x,t)=\{x\} \times \{t\} \times {h_X}^{-1}(s(t))$ {\rm (}$x \in E_{F_I}${\rm )}.
Suppose the following in addition.
\begin{itemize}
\item At each point $(x_0,t_0,y_0) \subset X_{(F_I,E_{F_I},s_I,h_X),m} \bigcap (\{x_0  \in E_{F_I}\} \times \{t_0\} \times {h_X}^{-1}(s(t_0)))$ where $s_I(t_0)$ is not a critical value of $h_X$, the value of the derivative $\frac{\partial F_I}{\partial t}$ at $(x,t)=(x_0,t_0)$ is not $0$.

\item At each point $(x_0,t_0,y_0) \subset X_{(F_I,E_{F_I},s_I,h_X),m} \bigcap (\{x_0  \in E_{F_I}\} \times \{t_0\} \times {h_X}^{-1}(s(t_0)))$ where $s_I(t_0)$ is a critical value of $h_X$, both the following hold.
\begin{itemize}
\item The value of the derivative $\frac{\partial F_I}{\partial t}$ at $(x,t)=(x_0,t_0)$ is not $0$.
\item The value of the derivative $\frac{\partial F_I}{\partial x}$ at $(x,t)=(x_0,t_0)$ is not $0$, for some $x_i$, where $x:=(x_1,\ldots x_n)$.
\item $s_I(t_0)$ is not a critical value of $s_I$.
\end{itemize}
\end{itemize}

In this situation, we have an $m$-dimensional smooth submanifold $X_{(F_I,E_{F_I},s_I,h_X),m}$  of ${\mathbb{R}}^{m+m_2-m_1+1}$ with no boundary.
\end{Thm}
\begin{proof}
The present exposition is essentially similar and a bit different from exposition and arguments from \cite{kitazawa5, kitazawa6, kitazawa8}.

We investiagte local coordinates for points of $X_{(F_I,E_{F_I},s_I,h_X),m}$.

\ \\
Case 2-A  At each point $(x_0,t_0,y_0) \subset X_{(F_I,E_{F_I},s_I,h_X),m} \bigcap (\{x_0 \} \times \{t_0\} \times {h_X}^{-1}(s(t_0)))$ where $s_I(t_0)$ is not a critical value of $h_X$,. \\

We investigate a small open neighborhood $U_{x_0}$ of $x_0$ in ${\mathbb{R}}^n$ with ${\overline{U_{x_0}}}^{{\mathbb{R}}^n}$ being compact. The value of the derivative $\frac{\partial F_I}{\partial t}$ at $(x,t)=(x_0,t_0)$ is not $0$. 
A proper height function $h_X:={{\pi}_{m_2,1}} {\mid}_{X}:X \rightarrow {\mathbb{R}}$ on an $m_1$-dimensional submanifold of ${\mathbb{R}}^{m_2}$ is given.

As in the proof of Theorem \ref{thm:1}, the restriction of the map ${\pi}_{m_2+n,n} {\mid}_{X_{(F_I,E_{F_I},s_I,h_X),m}}$ to the preimage ${{\pi}_{m_2+n,n}}^{-1}(U_{(x_0,s_I(t_0))})$ is regarded to be the projection of a smooth product bundle. Due to the fact that the value of the derivative $\frac{\partial F_I}{\partial t}$ at $(x,t)=(x_0,t_0)$ is not $0$, we have a smooth embedding $X_{(F_I,E_{F_I},s_I,h_X),m} \subset {\mathbb{R}}^{m+m_2-m_1+1}$ locally and canonically. \\
\ \\
Case 2-B At each point $(x_0,t_0,y_0) \subset X_{(F_I,E_{F_I},s_I,h_X),m} \bigcap (\{x_0\} \times \{t_0\} \times {h_X}^{-1}(s(t_0)))$ where $s_I(t_0)$ is a critical value of $h_X$. \\

We investigate a small open neighborhood $U_{x_0}$ of $x_0$ in ${\mathbb{R}}^n$ with ${\overline{U_{x_0}}}^{{\mathbb{R}}^n}$ being compact. The value of the derivative $\frac{\partial F_I}{\partial t}$ at $(x,t)=(x_0,t_0)$ is not $0$. The value of the derivative $\frac{\partial F_I}{\partial x}$ at $(x,t)=(x_0,t_0)$ is not $0$, for some $x_i$, where $x:=(x_1,\ldots x_n)$, and due to implicit function theorem, the intersection $U_{x_0} \bigcap {F_{I,t_0}}^{-1}(0)$ is a (regular) smooth submanifold of dimension $n-1$ in $U_{x_0}$ and has no boundary. $s_I(t_0)$ is not a critical value of $s_I$. From this, the restriction of the map ${\pi}_{m_2+n,n} {\mid}_{X_{(F_I,E_{F_I},s_I,h_X),m}}$ to the preimage ${{\pi}_{m_2+n,n}}^{-1}(U_{x_0})$ is $\mathcal{A}$-equivalent to the product map of a smooth function being the restrcition of $h_X$ to a suitable $m_1$-dimensional smooth submanifold of $X$ and the identity map on
 $U_{x_0} \bigcap {F_{I,t_0}}^{-1}(0)$. \\
\ \\
Case 2-A and Case 2-B complete the proof.
\end{proof}

\begin{Rem}
\label{rem:1}
In \cite{kitazawa8}, we consider a specific case which is of simplest cases. More precisely, the case is explained as the case of $h_X:={{\pi}_{m_1+1,1}} {\mid}_{S^{m_1}}:S^{m_1} \rightarrow \mathbb{R}$ or the height function ${\pi}_{m_1+1,1} {\mid}_{A_{S^{m_1}}}$ of $A_{S^{m_1}}$ obtained as the image $A_{S^{m_1}}$ of $S^{m_1}$ by an affine transformation of ${\mathbb{R}}^{m_1+1}$ ($m_2:=m_1+1$).

We do not discuss more explicit and precise description in the situation of Theorem \ref{thm:2}. This is left to readers.
\end{Rem}

\begin{Def}
\label{def:3}
In Theorem \ref{thm:2}, we say the quadruple $(F_I,E_{F_I},s_I,h_X)$ is a {\it very admissible system of a foliated region and a height of type $(n,m_1,m_2)$} or an {\it ${\rm (}n,m_1,m_2{\rm )}$-VA-FR-H-system}.
\end{Def}

We present Corollary \ref{cor:3}. Remark \ref{rem:1} can be regarded to be a specific case of Theorem \ref{thm:2} and its specific case for certain round fold maps is generalized and formulated as Corollary \ref{cor:3}. Corollary \ref{cor:3} is also essentially a kind of simply generalized versions of \cite[Theorem 3]{kitazawa8} (\cite[Example 2 or Example 3]{kitazawa8}). We omit its precise proof and the proof is left to readers as a kind of fundamental exercises.

\begin{Cor}
\label{cor:3}

\begin{enumerate}
\item Let $I:=\{t \geq 0\}$, $a>0$, $b \in \mathbb{R}$, $F_{I,t}(x):=a-t-{\Sigma}_{j=1}^2 {x_j}^2$ {\rm (}$x=(x_1,x_2)${\rm )}, and $s_I(t)=t+b$. Let $h_X$ be a smooth proper height function on a manifold $X$ satisfying the following.
\begin{itemize}
\item $h_X$ has the minimum $b$ and $s_I(I) \subset h_X(X)$.
\item The singular value set $h_X(S(h_X))$ of $h_X$ is closed and discrete in $\mathbb{R}$.
\item $a+b$ is not a critical value of $s_I$.
\end{itemize}
 We have an ${\rm (}n,m_1,m_2{\rm )}$-VA-FR-H-system $(F_I,E_{F_I},s_I,h_X)$.

\item Let $I:=\{a_1 \leq t \leq a_2\}$ with $a_2>a_1>0$, $b \in \mathbb{R}$, $F_{I,t}(x):=t^2-{\Sigma}_{j=1}^2 {x_j}^2$ {\rm (}$x=(x_1,x_2)${\rm )}, and $s_I(t)=(t-a_1)(a_2-t)+b$. Let $h_X$ be a smooth proper height function on a manifold $X$ satisfying the following.
\begin{itemize}
\item $h_X$ has the minimum $b$ and $s_I(I) \subset h_X(X)$.
\item The singular value set $h_X(S(h_X))$ of $h_X$ is closed and discrete in $\mathbb{R}$.
\item $\frac{{(a_2-a_1)}^2}{4}+b$ is not a critical value of $h_X$.
\end{itemize}
 We have an ${\rm (}n,m_1,m_2{\rm )}$-VA-FR-H-system $(F_I,E_{F_I},s_I,h_X)$.

\end{enumerate}
In both cases, we have normalized round fold maps ${\pi}_{m+m_2-m_1+1,n} {\mid}_{X_{(F_I,E_{F_I},s_I,h_X),m}}:X_{(F_I,E_{F_I},s_I,h_X),m} \rightarrow {\mathbb{R}}^n$ for $n \geq 2$.
\end{Cor}
We only note that $t=\frac{a_1+a_2}{2}$ is the unique critical point of $s_I$.
\subsection{Examples.}
We close the present paper by presenting related examples of certain new types.
\begin{Ex}
\label{ex:2}
Let $r>0$ be a number satisfying $\frac{\sqrt{3}}{2}<r<1$. Let $l \geq 0$ be a non-negative integer. 
We give circles $\{(x_1,x_3) \mid {(x_1-j)}^2+{(x_3+r)}^2=1\}$ indexed by all integers whose absolute values are smaller than or equal to $l$ or all positive integers whose absolute values are smaller than or equal to $l$. Any distinct two of them have no intersection in $\{(x_1,x_3) \mid x_3 \geq 0\}$.  
We also give two segments $\{(x_1,x_3) \mid x_1+x_3+l+1=0, 0 \leq x_3 \leq 1\}$ and $\{(x_1,x_3) \mid -x_1+x_3+l+1=0, 0 \leq x_3 \leq 1\}$.
We put $f_j(x_1,x_3):={(x_1-j)}^2+{(x_3+r)}^2-1$. We also put $f_{+}(x_1,x_3):=x_1+x_3+l+1$ and $f_{-}(x_1,x_3):=-x_1+x_3+l+1$.
Let $f_{0,\pm}(x_1,x_3)$ denote the product of all these functions.
We can define the set $\{(x_1,x_2,x_3) \mid  f_{0,\pm}(x_1,x_3)-{x_2}^2=0, 0 \leq x_3 \leq 1\}$. This is, by arguments of \cite[Main Theorem 1 and Proof of Main Theorem 1]{kitazawa5}, or due to implicit function theorem, essentially, we have a compact, conenected and smooth surface embedded in ${\mathbb{R}^3}$ and by composing the (restriction of the) projection ${\pi}_{3,2}$, we have a region in ${\mathbb{R}}^2$.

Let $I:=\{t \mid a_1=0 \leq t \leq a_2=1\}$, $b \in \mathbb{R}$, $F_{I,t}(x):=f_{0,\pm}(x_1,t)-{x_2}^2$ ($x=(x_1,x_2)$), and $s_I(t)=b+(t-a_1)(a_2-t)$.

Let $h_X$ be a smooth proper height function on a manifold $X$ satisfying the following.
\begin{itemize}

\item $h_X$ has the minimum $b$ and $s_I(I) \subset h_X(X)$.
\item The singular value set $h_X(S(h_X))$ of $h_X$ is closed and discrete in $\mathbb{R}$.
\item $s_l(1-r)$ is not a critical value of $h_X$.
\end{itemize}
We also pose the condition that $1-r$ is not a critical point of $s$. In short we assume $1-r \neq \frac{a_1+a_1}{2}$. 
We have an ${\rm (}n,m_1,m_2{\rm )}$-VA-FR-H-system $(F_I,E_{F_I},s_I,h_X)$.
\end{Ex}
\begin{Ex}
\label{ex:3}
Let $l>0$ be a positive integer. Let $\{p_j\}_{j=1}^l$ be an increasing sequence of real numbers. 
Let $f_{\{p_j\}_{j=1}^l}(x_1,x_3):=x_3-{\prod}_{j=1}^l {(x_1-p_j)}^2$. Let $T_{\{p_j\}_{j=1}^l}$ be a sufficiently large positive number. 
We can define the set $\{(x_1,x_2,x_3) \mid f_{\{p_j\}_{j=1}^l}(x_1,x_3)-{x_2}^2=0, x_3 \leq T_{\{p_j\}_{j=1}^l}\}$. This is, by arguments of \cite[Main Theorem 1 and Proof of Main Theorem 1]{kitazawa5}, or due to implicit function theorem, essentially, we have a compact, conenected and smooth surface embedded in ${\mathbb{R}^3}$ and by composing the (restriction of the) projection ${\pi}_{3,2}$, we have a region in ${\mathbb{R}}^2$.

Let $I:=\{t \mid t \geq a_0=-T_{\{p_j\}_{j=1}^l}\}$, $b \in \mathbb{R}$, $F_{I,t}(x):=f_{\{p_j\}_{j=1}^l}(x_1,t)-{x_2}^2$ ($x=(x_1,x_2)$), and $s_I(t)=t-a_0+b$.

Let $h_X$ be a smooth proper height function on a manifold $X$ satisfying the following. 
\begin{itemize}
\item $h_X$ has the minimum $b$ and $s_I(I) \subset h_X(X)$.
\item The singular value set $h_X(S(h_X))$ of $h_X$ is closed and discrete in $\mathbb{R}$.

\item $b+s_l(t_s)$ is not a critical value of $h_X$ for any critical value $t_s$ of the function ${\prod}_{j=1}^l {(x-p_j)}^2$ on $\mathbb{R} \ni x$.
\end{itemize}
We have an ${\rm (}n,m_1,m_2{\rm )}$-VA-FR-H-system $(F_I,E_{F_I},s_I,h_X)$.
\end{Ex}
Example \ref{ex:4} is related to ${\rm (}n,m_1,m_2{\rm )}$-VA-FR-H-systems. We present a quadruple $(F_I,E_{F_I},s_I,h_X)$ which is similar to an ${\rm (}n,m_1,m_2{\rm )}$-VA-FR-H-system and which is not an ${\rm (}n,m_1,m_2{\rm )}$-VA-FR-H-system, as Example \ref{ex:4}.
\begin{Ex} 
\label{ex:4}
 Let $I:=\{a_1 \leq t \leq a_2\}$ with $a_2>a_1>0$, $b \in \mathbb{R}$, $F_{I,t}(x):=x_1-tx_2^2-t^3$ ($F_{I,t}:{\mathbb{R}}^2 \rightarrow \mathbb{R}$) and $s_I(t):=(t-a_1)(a_2-t)+b$. 
The value of the derivative $\frac{\partial F_{I,t}}{\partial t}$ at $(x,t):=(x_0,t_0)=(x_{0,1},x_{0,2},t_0)$ is ${x_{0,2}}^2-3t^2$ and $0$ at $(x_0,t_0)=(0,0)$.
Let $h_X$ be a smooth proper height function.
\begin{itemize}
\item $h_X$ has the minimum $b$ and $s_I(I) \subset h_X(X)$.
\item The singular value set $h_X(S(h_X))$ of $h_X$ is closed and discrete in $\mathbb{R}$.
\end{itemize}
 We have a quadruple $(F_I,E_{F_I},s_I,h_X)$ which is similar to a ${\rm (}2,m_1,m_2{\rm )}$-VA-FR-H-system and which is not a  ${\rm (}2,m_1,m_2{\rm )}$-VA-FR-H-system.
\end{Ex}

\section{Conflict of interest. Data availability. AI.}
 The author is a researcher at Osaka Central Advanced Mathematical Institute (OCAMI researcher) and it is funded by MEXT Promotion of Distinctive Joint Research Center Program JPMXP0723833165. He thanks all there for the hospitality, where he is not employed by the institute or related projects 
 
We have no data other than the present file related to the present paper. We do not assume non-trivial arguments in formally unpublished. It is no problem to refer to them, to some extent.

Essentially, AI is not used for our mathematical arguments and result.

\end{document}